%% file: main_part_FMRC.tex
\documentclass[hidelinks,onefignum,onetabnum]{siamart220329}

\input{ex_shared}
\DeclareMathOperator{\divg}{div}
\usepackage{verbatim}
\usepackage{subcaption}
\usepackage{float}
\usepackage{placeins}

\ifpdf
\hypersetup{
  pdftitle={Optimal Low-dimensional Approximation of Transfer Operators via Flow Matching: Computation and Error Analysis},
  pdfauthor={Zhicheng Zhang, Ling Guo, and Hao Wu}
}
\fi

\begin{document}

\maketitle
\renewcommand{\qedsymbol}{}

\begin{abstract}
Reaction coordinates (RCs) are low-dimensional representations of complex dynamical systems that capture their long-term dynamics. In this work, we focus on the criteria of lumpability and decomposability, previously established for assessing RCs, and propose a new flow matching approach for the analysis and optimization of reaction coordinates based on these criteria. This method effectively utilizes data to quantitatively determine whether a given RC satisfies these criteria and enables end-to-end optimization of the reaction coordinate mapping model. Furthermore, we provide a theoretical analysis of the relationship between the loss function used in our approach and the operator error induced by dimension reduction.
\end{abstract}

\begin{keywords}
reaction coordinate, transfer operator, error bound, flow matching, dimension reduction 
\end{keywords}

\section{Introduction}
The reduction of model complexity and the simulation of long-term behavior in complex, high-dimensional dynamical systems are essential across various scientific disciplines. With the increasing computational power, it is now feasible to directly simulate a wide range of complex systems, from macro-scale phenomena such as climate to micro-scale systems like fluid dynamics and protein folding. These simulations provide valuable trajectory data for further analysis and prediction. To investigate system behavior over time scales longer than those accessible through direct simulation, a variety of model reduction methods have been developed (e.g., \cite{mardt2018vampnets,wu2020variational,noe2013variational,williams2015data}).
In these contexts, it is observed that over long timescales, these systems often exhibit reduced complexity and can be effectively characterized by a small set of features, known as reaction coordinates (RCs). By projecting raw data onto subspaces parameterized by these RCs, the complexity of the system can be significantly reduced while retaining essential dynamical information, enabling further in-depth analysis.

In the identification of RCs, methods based on the spectral decomposition of transfer operators have long played a crucial role \cite{noe2013variational,nuske2016variational,mardt2018vampnets,chen2019nonlinear,wu2020variational}. However, these methods typically yield linear reduced-order kinematic models. For systems where the transfer density functions reside on nonlinear low-dimensional manifolds \cite{bittracher2018transition}, these approaches often struggle to identify lower-dimensional RCs. To address this limitation, various deep learning-based methods have been proposed in recent years, including techniques such as time-lagged autoencoders \cite{wehmeyer2018time}, information bottlenecks \cite{wang2019past,wang2021state,federici2023latent}, and normalizing flows \cite{wu2024reaction}. While these approaches offer promising avenues, they lack the rigorous theoretical foundation that underpins transfer operator-based methods, particularly in establishing a clear relationship between the loss functions used in deep learning and the operator errors induced by dimension reduction.

In \cite{bittracher2023optimal}, the authors extended conclusions from finite-state Markov chains to general Markov processes, theoretically establishing two criteria for evaluating the quality of RCs: \textit{lumpability} and \textit{decomposability}. These criteria define the conditions that the reduced-order transition densities, derived from optimal RCs, should satisfy. However, the numerical methods proposed in that work involve the estimation and integration of system transition probability densities, which limits scalability and makes them challenging to apply to high-dimensional systems.

In this paper, we recognize that the tasks of evaluating and optimizing RCs based on lumpability and decomposability can be equivalently reformulated as the analysis and refinement of two conditional probability models. Leveraging this equivalence, we introduce a novel and highly effective technique from generative modeling—flow matching—to tackle the data-driven identification of RCs. This approach, termed FMRC, offers a scalable solution for identifying and refining RCs.

Furthermore, we establish a theoretical framework that links lumpability and decomposability to the reduction of transfer operators. Based on this framework, we demonstrate that the loss function used in FMRC training can theoretically control the error between the reduced-order and original transition operators. Compared to existing deep learning-based methods, FMRC provides a more robust and theoretically grounded approach to reaction coordinate identification.


\section{Preliminaries}
\label{sec:pre}
Before delving into our model reduction method, we first introduce some relevant definitions, notations, and assumptions that will be used throughout this work.

\subsection{Markov processes and transition densities}
Let $\mathbb{X}\subset \mathbb{R}^D$ be a measurable set with Lebesgue measure $\lambda$. Consider a Markov process $\{X_t\}_{t\in\mathbb{R}^+}$, such as molecular dynamics data generated from an overdamped Langevin equation. Given a lag time $\tau$, the dynamics of $X_t$ can be characterized by the \textit{transition probability function} $P_\tau\colon \mathbb{X}\times\mathcal{B}\to[0,1]$, defined as
\begin{equation}
    P_\tau(x, B) = \operatorname{Prob}\left[X_{t+\tau} \in B \mid X_{t}=x\right] \quad \text{for all } t \geq 0,
\end{equation}
where $\mathcal{B}$ denotes the Borel $\sigma$-algebra on $\mathbb{X}$, and $B \in \mathcal{B}$. Here, $P_\tau(x, \cdot)$ is a probability measure on $\mathcal{B}$, and $P_\tau(\cdot,B)$ is a Borel-measurable function. We further assume that $P_\tau(x, \cdot)$ is absolutely continuous with respect to the Lebesgue measure, allowing us to express $P_\tau(x, \cdot)$ as
\begin{equation}
    P_\tau(x, B) = \int_B p_\tau(x,y)\, \mathrm{d}y, \quad \forall x \in \mathbb{X},\ B \in \mathcal{B},
\end{equation}
where $p_\tau\colon \mathbb{X}\times\mathbb{X}\to \mathbb{R}^+$ is the \textit{transition density}.

\subsection{Simulation data}
In this work, the simulation data for the Markov process $x_t$ are assumed to be recorded in the following form: $\{(x_n, y_n)\}_{n=1}^N$, representing all transition pairs $(x_t, x_{t+\tau})$ with a lag time $\tau$ observed in simulation trajectories. For example, given a trajectory $\{x_1, x_2, \dots, x_T\}$, the dataset would consist of the pairs $\{(x_1, x_{1+\tau}),$ $ (x_2, x_{2+\tau}),\ldots, (x_{T-\tau}, x_T)\}$.

Furthermore, we assume that as $N$ approaches infinity, the empirical joint probability density function of $x_n$ and $y_n$ is given by $\rho(x,y) = \rho_0(x) p_\tau(x,y)$, where $\rho_0$ represents the marginal probability density function of $x_n$. The corresponding marginal probability density function of $y_n$ is then given by
\[
\rho_1(y) = \int_{\mathbb{X}} \rho_0(x) p_\tau(x, y) \, \mathrm{d}x.
\]
It is important to note that we do not assume the system has reached equilibrium, meaning $\rho_0$ and $\rho_1$ may differ. In the case of equilibrium, we would have $\rho_0 = \rho_1 = \pi$, where $\pi$ denotes the stationary distribution of the process $\{X_t\}$.

\subsection{Transfer operators}

A fundamental tool for describing the global kinetic behavior of a metastable system is transfer operators. In this work, we focus on two key transfer operators: the Perron-Frobenius (PF) operator and the Koopman operator. These operators respectively characterize the evolution of the state distribution and the dynamics of observables within a Markov process.

\begin{definition}
For a given lag time $\tau$, the following operators are defined:
\begin{enumerate}
    \item Let $u_t(x) = p_t(x)/\rho_0(x) \in L_{\rho_0}^1(\mathbb{X})$ be a probability density with respect to the density $\rho_0$, where $L_{\rho_{0}}^{1}(\mathbb{X})=\{u|\int_{\mathbb{X}}u(x)\rho_{0}(x)\mathrm{d}x<\infty\}$. The Perron–Frobenius operator $\mathcal{T}_\tau: L_{\rho_0}^1(\mathbb{X})\to L_{\rho_1}^1(\mathbb{X})$ is given by
    \[
        (\mathcal{T}_\tau u_t)(y) = \int_{\mathbb{X}} \frac{\rho_0(x)}{\rho_1(y)} p_\tau(x,y) u_t(x) \, \mathrm{d} x=\mathbb E_{(X_t,X_{t+\tau})\sim\rho}\left[u_t(X_t)|X_{t+\tau}=y\right].
    \]
    
    \item The Koopman operator $\mathcal{K}_\tau: L^\infty(\mathbb{X})\to L^\infty(\mathbb{X})$ is defined as
    \[
        \mathcal{K}_\tau f(x) = \int_{\mathbb{X}} p_\tau(x, y) f(y) \, \mathrm{d} y = \mathbb E_{(X_t,X_{t+\tau})\sim\rho}\left[f\left(X_{t+\tau}\right) \mid X_t = x\right].
    \]
\end{enumerate}
\end{definition}
It is important to note that, unlike the common definition (see, e.g., Definition 2.5 in [1]), we consider the PF operator with respect to the empirical distribution of the data rather than the equilibrium distribution. This is particularly advantageous for both algorithm design and theoretical analysis when the simulation data is non-equilibrium. If the simulation data reach equilibrium, our definition of the PF operator coincides with the standard one.

It can be established that transfer operators are linear, enabling low-rank approximations through their dominant eigenfunctions or singular functions. Consequently, numerous dimension reduction methods based on variational estimation of transfer operators have been proposed, where dominant spectral components are selected as RCs. However, in systems with a large number of metastable states, the number of dominant spectral components usually corresponds to the number of these states. This inherent characteristic implies that such methods may face challenges in deriving low-dimensional and effective RCs.

\section{Criteria for optimal reaction coordinates}
\label{sec:criteria}
Any function $r:\mathbb{X} \to \mathbb{R}^d$ with $d \ll D$ can be considered a reaction coordinate (RC). In this paper, we consider $r$ to be continuously differentiable. Furthermore, for any $z \in \mathbb{Z}$, the level set $\Sigma_{r}(z):=\{x \in \mathbb{X} \mid r(x)=z\}$ is assumed to be a $(D - d)$-dimensional topological submanifold of $\mathbb{X}$.
\cite{bittracher2023optimal} proposes that optimal RCs should approximately satisfy two key conditions: \textit{lumpability} and \textit{decomposability}. We will now introduce these two conditions in detail.

\begin{definition}[Lumpability]\label{def:lump}
    A system is said to be lumpable with respect to the RC $r$ if, for all $x, y \in \mathbb{X}$, there exists a function $p_\tau^L\colon \mathbb{Z} \times \mathbb{X} \to \mathbb{R}^+$ such that
    \begin{equation}\label{eq:lump}
        p_\tau(x, y) = p_\tau^L(r(x), y).
    \end{equation}
\end{definition}

From the above definition, we can conclude that lumpability means that the transition density $p_\tau(x,\cdot)$ from a given $x$ depends only on the value of the RC $r(x)$, and not on the precise location within the level set of 
$r(x)$. If an RC satisfies the lumpability condition, we can infer that it effectively preserves and describes the long-term behavior of the system.

\begin{definition}[Decomposability]
    A system is said to be decomposable with respect to the RC $r$ if, for all $x, y \in \mathbb{X}$, there exists functions $p_\tau^D\colon \mathbb{X} \times \mathbb{Z} \to \mathbb{R}^+$ and $p_{\mathrm{local}}\colon \mathbb{Z} \times \mathbb{X} \to \mathbb{R}^+$ such that
    \begin{equation}\label{eq:decomp}
        p_\tau(x,y)=p_{\tau}^{D}(x,r(y))p_{\mathrm{local}}(r(y),y).
    \end{equation}
\end{definition}

The decomposability implies that the conditional distribution of $X_{t+\tau}$ given the system state $X_t$ can be determined through a two-step procedure. First, the distribution of the RC $r\left(X_{t+\tau}\right)$ at time $t+\tau$ is computed by $p_\tau^D$. Next, within the level set defined by each value of the RC, the distribution of $X_{t+\tau}$ is evaluated by using $p_{\mathrm{local}}$, which is independent of $X_t$.

\begin{remark}
The definition presented here differs slightly in form from that in \cite{bittracher2023optimal}. In \cite{bittracher2023optimal}, the decomposability is defined as follows:
\begin{equation}\label{eq:original-decomp}
p_{\tau}(x,y) = \tilde{p}_{\tau}^{D}(x,r(y))\pi(y).
\end{equation}
However, it can be shown that under the existence of the stationary distribution, the two definitions are equivalent. Utilizing the properties of the stationary distribution, we can derive from \eqref{eq:decomp} that
\[
\pi(y) = p_{\mathrm{local}}(r(y),y)\cdot\int\pi(x)p_{\tau}^{D}(x,r(y))\mathrm{d}x,
\]
which leads to the expression for decomposability as given in \eqref{eq:original-decomp}, where
\[
\tilde{p}_{\tau}^{D}(x,r)=\frac{p_{\tau}^{D}(x,r)}{\int\pi(x)p_{\tau}^{D}(x,r)\mathrm{d}x}.
\]
\end{remark}

It can be observed that the definition of decomposability is relatively complex. To simplify, we introduce a backward transition density:
\[
p_{-\tau}(y, x) = \frac{\rho_{0}(x) p_{\tau}(x, y)}{\rho_{1}(y)},
\]
where \( p_{-\tau}(y, x) \) represents the posterior probability density of \( X_t = x \) given \( X_{t+\tau} = y \), assuming the prior distribution of \( X_t \) is \( \rho_0 \). We can then provide an equivalent characterization of d decomposability:
\begin{lemma}\label{lem:simple-decomp}
A system is decomposable with respect to $r$ if and only if there exists a function \( p_{-\tau}^D\colon \mathbb{Z} \times \mathbb{X} \to \mathbb{R}^+ \) such that
\begin{equation}\label{eq:simple-decomp}
p_{-\tau}(y, x) = p_{-\tau}^D(r(y), x).
\end{equation}
\end{lemma}
\begin{proof}
We first prove the necessity. Assuming that \eqref{eq:decomp} holds, and using the definition of $p_{-\tau}$, we have
\begin{eqnarray*}
p_{-\tau}(y, x) & = & \frac{\rho_{0}(x) p_{\tau}(x, y)}{\int_{\mathbb{X}} \rho_{0}(x) p_{\tau}(x, y) \, \mathrm{d}x} \\
 & = & \frac{\rho_{0}(x) \tilde{p}_{\tau}^{D}(x, r(y)) p_{\mathrm{local}}(r(y), y)}{\int_{\mathbb{X}} \rho_{0}(x) \tilde{p}_{\tau}^{D}(x, r(y)) p_{\mathrm{local}}(r(y), y) \, \mathrm{d}x} \\
 & = & \frac{\rho_{0}(x) \tilde{p}_{\tau}^{D}(x, r(y))}{\int_{\mathbb{X}} \rho_{0}(x) \tilde{p}_{\tau}^{D}(x, r(y)) \, \mathrm{d}x}.
\end{eqnarray*}
Note that the right-hand side depends only on $x$ and $r(y)$, thereby establishing the necessity.

Next, we prove the sufficiency. Assuming that \eqref{eq:simple-decomp} holds, we have
\begin{eqnarray*}
p_{\tau}(x, y) & = & \frac{\rho_{1}(y) p_{-\tau}(y, x)}{\int_{\mathbb{X}} \rho_{1}(y) p_{-\tau}(y, x) \, \mathrm{d}y} \\
 & = & \frac{p_{-\tau}^{D}(r(y), x)}{\int_{\mathbb{X}} \rho_{1}(y) p_{-\tau}^{D}(r(y), x) \, \mathrm{d}y} \cdot \rho_{1}(y).
\end{eqnarray*}
Thus, the sufficiency is established.
\end{proof}

Based on the analysis above, we can infer that lumpability implies $r(X_t)$ is a sufficient descriptor for predicting the future state $X_{t+\tau}$ based on the state at time $t$. Similarly, decomposability implies that $r(X_{t+\tau})$ serves as a sufficient statistics for estimating the past state $X_t$ for given $X_{t+\tau}$. However, in practical systems, these criteria are typically not fully satisfied. To address this, we can assess the quality of the RC by quantifying errors in the low-dimensional representations of the transition densities given in \eqref{eq:lump} and \eqref{eq:simple-decomp}. These errors can then be minimized to further optimize the RC. In the following, we will use the flow matching approach to perform both evaluation and optimization.

\section{Flow matching for reaction coordinates}
\label{sec:FMRC}
\subsection{Fundamentals of flow matching}
Flow matching (FM) is an important flow-based generative model that effectively constructs a mapping from a simple distribution to a complex target distribution. There are various variants of FM \cite{albergo2023stochastic,lipman2022flow,liu2022rectified,tong2023improving}, but here we focus on the most basic one, also known as rectified flow \cite{liu2022rectified}.

Assume $Y^0$ is an initial random variable following a simple distribution (such as the standard Gaussian distribution), and $Y^1$ is a random variable following the complex target distribution (the data distribution). We can define a random path $Y^s$ from $Y^0$ to $Y^1$ as \footnote{To distinguish this from the physical time \( t \) in dynamical systems, we use the superscript \( s \) to denote the virtual time involved in FM.}
\[
Y^s = I(s;Y^0,Y^1) = (1-s)Y^0 + sY^1, \quad \text{for } s \in [0,1].
\]
This path is non-causal because the derivative \( \partial I(s;Y^0,Y^1)/\partial s = Y^1 - Y^0 \) requires knowledge of \( Y^1 \).
To overcome the non-causality, we can introduce the expected velocity field of \( I(s; Y^0, Y^1) \):
\[
v^*(s,y) = \mathbb{E}\left[Y^1 - Y^0 \mid Y^s = y\right],
\]
where \( v^*(y,s) \) represents the mean velocity of all random paths passing through the point \( y \) at time \( s \). Using this velocity field, we can construct the following deterministic, causal flow:
\[
\frac{\mathrm{d}}{\mathrm{d}s} Y^s = v^*(s, Y^s),
\]
An important result of FM is that when the initial value is set as the random variable \( Y^0 \), the final value obtained from this flow at $s=1$ has a distribution identical to the target distribution of $Y^1$.

In practical applications, a parameterized model (e.g., a neural network) $v$, can be used to approximate $v^*$, and the model can be trained by minimizing the following loss function:
\[
\int_0^1 \mathbb E\left[\Vert v(s, Y^s) - (Y^1 - Y^0) \Vert^2\right] \, \mathrm{d}s.
\]

The above result can also be extended to the modeling of conditional probability distributions. Suppose \( Y^1 \) and \( X \) are two dependent random variables, and we assume that the conditional distribution \( Y^0|X \) of the initial variable \( Y^0 \) is tractable.\footnote{A common choice for \( Y^0 \) in practice is a random variable following a standard Gaussian distribution, independent of both \( X \) and $Y^1$.} We can then solve the following optimization problem:
\[
v^{*} = \arg\min_{v}\int_{0}^{1}\mathbb{E}\left[\Vert v(s, Y^{s}, X) - (Y^{1} - Y^{0})\Vert^{2}\right]\,\mathrm{d}s.
\]
For any given \( X = x \), the optimal \( v^* \) can be used to transform the distribution of \( Y^0 \) into the conditional distribution \( Y^1|X = x \) through the flow defined by \( \mathrm{d}Y^s/\mathrm{d}s = v^*(s, Y^s, x) \). A detailed analysis of flow matching-based modeling for conditional distributions can be found in Subsection \ref{subsec:fmrc-error}.

\subsection{FMRC approach}\label{subsec:FMRC}
Building on the previous introduction to the FM model, we can accurately recover the densities $p_\tau$ and $p_{-\tau}$ using the vector fields $v_0^*, v_1^*:\mathbb R\times\mathbb R^D\times\mathbb R^D\to\mathbb R^D$, which are the solutions to the optimization problem:
\begin{equation}\label{eq:full-loss}
    \min _{v_0, v_1} \mathcal L[v_0,v_1]=\mathcal{L}_0\left[v_0\right] + \mathcal{L}_1\left[v_1\right],
\end{equation}
where
\begin{eqnarray*}
\mathcal{L}_{0}[v] & = & \mathbb{E}_{(X_{t},X_{t+\tau})\sim\rho,X_{t+\tau}^{\prime}\sim\mathcal{N}(\mathbf{0},\mathbf{I})}\left[\left\Vert v\left(s, X_{t+\tau}^{s}, X_{t}\right) - \left(X_{t+\tau} - X_{t+\tau}^{\prime}\right)\right\Vert ^{2}\right], \\
\mathcal{L}_{1}[v] & = & \mathbb{E}_{(X_{t},X_{t+\tau})\sim\rho,X_{t}^{\prime}\sim\mathcal{N}(\mathbf{0},\mathbf{I})}\left[\left\Vert v\left(s, X_{t}^{s}, X_{t+\tau}\right) - \left(X_{t} - X_{t}^{\prime}\right)\right\Vert ^{2}\right],
\end{eqnarray*}
$X_{t+\tau}^s = I\left(s, X_{t+\tau}', X_{t+\tau}\right)$, $X_{t}^s = I\left(s, X_{t}', X_{t}\right)$, and $\mathcal N(\mathbf 0,\mathbf I)$ denotes the standard Gaussian distribution.

It is worth noting that if the lumpability and decomposability conditions defined in Definition \ref{def:lump} and Lemma \ref{lem:simple-decomp} hold strictly with respect to a given RC $r$, then the conditional distributions $X_{t+\tau}|X_t$ and $X_t|X_{t+\tau}$ are identical to $X_{t+\tau}|r(X_t)$ and $X_t|r(X_{t+\tau})$, respectively. This implies that we can impose the bottlenecks defined by $r$ in the vector fields of the flow models for $p_\tau$ and $p_{-\tau}$, and solve the following problem:
\begin{equation}\label{eq:reduced-loss}
\min_{v_0^{\mathrm{RC}}, v_1^{\mathrm{RC}}} \mathcal{L}^{\mathrm{FMRC}}[v_0^{\mathrm{RC}}, v_1^{\mathrm{RC}},r] = \mathcal{L}_0\left[v_0^{\mathrm{RC}}(\cdot, \ast, r(\#))\right] + \mathcal{L}_1\left[v_1^{\mathrm{RC}}(\cdot, \ast, r(\#))\right],
\end{equation}
where $v_0^{\mathrm{RC}}, v_1^{\mathrm{RC}}: \mathbb{R} \times \mathbb{R}^D \times \mathbb{R}^d\to\mathbb R^d$. It can be seen that the resulting flow model actually estimates the conditional distributions $X_{t+\tau}|r(X_t)$ and $X_t|r(X_{t+\tau})$.

For the loss functions defined in \eqref{eq:full-loss} and \eqref{eq:reduced-loss}, we can establish the following result:
\begin{theorem}
Given an RC $r$,
\begin{enumerate}
\item $\min_{v_0^{\mathrm{RC}}, v_1^{\mathrm{RC}}} \mathcal{L}^{\mathrm{FMRC}}[v_0^{\mathrm{RC}}, v_1^{\mathrm{RC}},r] \ge \min _{v_0, v_1} \mathcal L[v_0,v_1]$,
\item The equality holds if and only if lumpability and decomposability are satisfied.
\end{enumerate}
\end{theorem}
\begin{proof}

The first conclusion is trivial to prove, so we will focus on proving the second conclusion.

First, let us assume that the equality in Conclusion 1 holds. This implies that the optimal velocity fields are given by
\[
v_{0}^{\mathrm{RC},*}\left(\cdot, \ast, r(\#)\right) = v_{0}^{*}\left(\cdot, \ast, \#\right),\quad v_{1}^{\mathrm{RC},*}\left(\cdot, \ast, r(\#)\right) = v_{1}^{*}\left(\cdot, \ast, \#\right).
\]
Consequently, we can use the following flow models to obtain the conditional distributions \( X_{t+\tau} \mid X_{t} = x \) and \( X_t \mid X_{t+\tau} = y \):
\begin{equation*}
\frac{\mathrm{d}X_{t+\tau}^{s}}{\mathrm{d}s} = v_{0}^{\mathrm{RC},*}\left(s, X_{t+\tau}^{s}, r(x)\right), \quad \frac{\mathrm{d}X_{t}^{s}}{\mathrm{d}s} = v_{1}^{\mathrm{RC},*}\left(s, X_{t}^{s}, r(y)\right).
\end{equation*}
Let \( p_\tau^L(r(x), \cdot) \) and \( p_{-\tau}^D(r(y), \cdot) \) denote the conditional distributions defined by these flow models. It follows that lumpability and decomposability hold.

Next, suppose that lumpability and decomposability are satisfied. Then we have:
\begin{eqnarray*}
v_{0}^{*}(s, y^{s}, x) & = & \mathbb{E}\left[X_{t+\tau} - X_{t+\tau}^{\prime} \mid X_{t+\tau}^{s} = y^{s}, X_{t} = x\right] \\
 & = & \mathbb{E}\left[X_{t+\tau} - X_{t+\tau}^{\prime} \mid X_{t+\tau}^{s} = y^{s}, r(X_{t}) = r(x)\right] \\
 & = & v_{0}^{\mathrm{RC},*}(s, y^{s}, r(x)),
\end{eqnarray*}
and
\begin{eqnarray*}
v_{1}^{*}(s, x^{s}, y) & = & \mathbb{E}\left[X_{t} - X_{t}^{\prime} \mid X_{t}^{s} = x^{s}, X_{t+\tau} = y\right] \\
 & = & \mathbb{E}\left[X_{t} - X_{t}^{\prime} \mid X_{t}^{s} = x^{s}, r(X_{t+\tau}) = r(y)\right] \\
 & = & v_{1}^{\mathrm{RC},*}\left(s, x^{s}, r(y)\right).
\end{eqnarray*}
Therefore, the equality in Conclusion 1 holds.
\end{proof}

Based on the theorem above, for a given RC \( r \), we can use \(\mathcal{L}^{\mathrm{FMRC}}[v_{0}^{\mathrm{RC},*}, v_{1}^{\mathrm{RC},*}, r]\) to assess whether it approximately satisfies lumpability and decomposability. Furthermore, by minimizing \(\mathcal{L}^{\mathrm{FMRC}}\), we can obtain the optimal RC \( r^* \) for the system. In practical applications, we can approximate \(v_{0}^{\mathrm{RC},*}, v_{1}^{\mathrm{RC},*}, r^*\) using parameterized models, such as neural networks, and solve the following optimization problem to find an approximate optimal RC:
\[
\min_{\theta} \mathcal{L}^{\mathrm{FMRC}}[v_{0,\theta}^{\mathrm{RC}}, v_{1,\theta}^{\mathrm{RC}}, r_{\theta}],
\]
where \(\theta\) denotes the model parameters.
In this paper, we refer to this method as \textbf{FMRC}.

The stochastic gradient descent-based training process of FMRC using the data set \(\{(x_n, y_n)\}_{n=1}^N\) is summarized as follows:
\begin{enumerate}
    \item Draw a random mini-batch \(\{(x_b, y_b)\}_{b=1}^B\) from the data set.
    \item For each \(b\), sample \(x_b^\prime, y_b^\prime \sim \mathcal{N}(\mathbf 0,\mathbf I)\) and \(s \sim \mathcal U(0,1)\).
    \item For each \(b\), compute \(x_b^s = I(s, x_b^\prime, x_b)\) and \(y_b^s = I(s, y_b^\prime, y_b)\).
    \item Calculate \(\hat{\mathcal{L}}^{\mathrm{FMRC}} = \hat{\mathcal{L}}_0 + \hat{\mathcal{L}}_1\), where
    \begin{eqnarray*}
    \hat{\mathcal{L}}_0 & = & \frac{1}{B}\sum_{b=1}^{B}\left\Vert v_{0,\theta}^{\mathrm{RC}}(s, y_b^s, r_\theta(x_b)) - (y_b - y_b^\prime)\right\Vert^2, \\
    \hat{\mathcal{L}}_1 & = & \frac{1}{B}\sum_{b=1}^{B}\left\Vert v_{1,\theta}^{\mathrm{RC}}(s, x_b^s, r_\theta(y_b)) - (x_b - x_b^\prime)\right\Vert^2.
    \end{eqnarray*}
    \item Update \(\theta\) as \(\theta \gets \theta - \eta \frac{\partial \hat{\mathcal{L}}}{\partial \theta}\), where \(\eta\) is the learning rate.
    \item Repeat Steps 1-5 until convergence.
\end{enumerate}

\remark The FMRC method focuses on obtaining a reaction coordinate modeled by an encoder, so generative purposes are not our primary concern. However, we can still generate data by solving the flow ODEs through our trained velocity fields. 

\section{Error analysis of FMRC}
\label{sec:error}
In this section, we analyze the error of FMRC from the perspective of transfer operators.

\subsection{Reduced-order transfer operators}
Here we first introduce the concept of reduced-order transfer operators.
\begin{definition}[Reduced-Order Transfer Operators]
For an RC \( r \), the reduced PF operator \(\mathcal{T}_D\) and the reduced Koopman operator \(\mathcal{K}_L\) are defined as follows:
\begin{eqnarray*}
\mathcal{K}_L h(x) & = & \mathbb{E}_{(X_t, X_{t+\tau}) \sim \rho}\left[h(X_{t+\tau}) \mid r(X_t) = r(x)\right], \\
\mathcal{T}_D h(x) & = & \mathbb{E}_{(X_t, X_{t+\tau}) \sim \rho}\left[h(X_t) \mid r(X_{t+\tau}) = r(x)\right].
\end{eqnarray*}
\end{definition}

According to the above definition, for any function \( h \), there exist functions \( g_L \) and \( g_D \) such that \(\mathcal{K}_L h(\cdot) = g_L(r(\cdot))\) and \(\mathcal{T}_D h(\cdot) = g_D(r(\cdot))\). In other words, for any function \( h: \mathbb{R}^D \to \mathbb{R} \), the action of the reduced-order transfer operators transforms it into a function that is intrinsically defined on the reaction coordinate space \(\{r(x) \mid x \in \mathbb{X}\}\).

We can then express lumpability and decomposability in terms of operators:
\begin{lemma}
Lumpability holds if and only if \(\mathcal{K}_L = \mathcal{K}_\tau\), and decomposability holds if and only if \(\mathcal{T}_D = \mathcal{T}_\tau\).
\end{lemma}
\begin{proof}
The result follows directly from the definition of the reduced-order transfer operators and is therefore omitted.
\end{proof}

It should be noted that FMRC provides approximate reduced-order transfer operators, which are defined as:
\begin{eqnarray*}
\hat{\mathcal{K}}_L h(x) & = & \mathbb{E}\left[h(\hat{X}_{t+\tau}) \mid r(X_t) = r(x)\right], \\
\hat{\mathcal{T}}_D h(x) & = & \mathbb{E}\left[h(\hat{X}_t) \mid r(X_{t+\tau}) = r(x)\right],
\end{eqnarray*}
where \(\hat{X}_{t+\tau}\) and \(\hat{X}_t\) are generated by the following flow models:
\begin{eqnarray}
\frac{\mathrm{d}X_{t+\tau}^{s}}{\mathrm{d}s} & = & v_{0,\theta}^{\mathrm{RC}}(s,X_{t+\tau}^{s},r(x)),\label{eq:flow-K}\\
\frac{\mathrm{d}X_{t}^{s}}{\mathrm{d}s} & = & v_{1,\theta}^{\mathrm{RC}}(s,X_{t}^{s},r(y)).
\end{eqnarray}
Therefore, by comparing \(\hat{\mathcal{K}}_L\) and \(\hat{\mathcal{T}}_D\) with the full-state transfer operators \(\mathcal{K}_\tau\) and \(\mathcal{T}_\tau\), we can analyze the impact of FMRC-induced errors.

\subsection{Wasserstein-2 distance}
For the sake of analysis, we briefly introduce the Wasserstein-2 distance and its relevant properties according to \cite{peyre2018comparison}.

\begin{definition}[Quadratic Wasserstein distance, or $W_2$ distance]
    Let $M$ be a connected Riemannian manifold with its distance $d(\cdot,\cdot)$. For two probability measures $u, v$ on $M$, denoting by $\Pi(u,v)$ the set of all couplings of $u$ and $v$, for $\gamma \in\Pi(u,v)$ one defines
    \[
    I(\pi):=\int_{M \times M} d(x, y)^2 \gamma(\mathrm{d} x, \mathrm{d} y)
    \]
    and then
    \[
    \mathrm{W}_2(\mu,\nu):=\inf \{I(\gamma) \mid \gamma \in \Pi(\mu,\nu)\}^{1 / 2}
    \]
    $W_2$ is a (possibly infinite) distance called the quadratic Wasserstein distance, or simply $W_2$ distance.
\end{definition}

The $W_2$ distance is a metric defined between probability distributions and is closely related to the optimal transport problem (see, e.g., \cite{villani2009optimal}). Despite its significance, the $W_2$ distance is nonlinear and therefore challenging to learn. However, under infinitesimal perturbations, the $W_2$ distance exhibits linearized behavior. To illustrate this in detail, we first define a semi-norm in the Sobolev space.

\begin{definition}    
    If $m$ is a positive measure on $M$, $f$ is a $C^1$ real function on $M$, one denotes
    \[
    \|f\|_{\dot{\mathrm{H}}^1(m)}:=\left(\int_M|\nabla f(x)|^2 m(\mathrm{d} x)\right)^{1 / 2}
    \]
    which defines a semi-norm; for $\nu$ a signed measure on $M$, one then denotes
    \[
    \|\nu\|_{\dot{H}^{-1}(m)}\colon =\sup \left\{|\langle f, \nu\rangle| \mid\|f\|_{\dot{H}^1(m)} \leq 1\right\},
    \]
    where the duality product $\left<f,\nu\right>$ denotes the integral of the function $f$ against the measure $\nu$. We observe that $\|\cdot\|_{\dot{H}^{-1}(m)}$ defines a (possibly infinite) norm, which we will call the $\dot{H}^{-1}(m)$ \textit{weighted homogeneous Sobolev norm}. In this article we use the standard Lebesgue measure on $\mathbb{X}$ so we denote $\dot{H}^{-1}(m)$ simply as $\dot{H}^{-1}$.
\end{definition}

The \textit{Benamou-Brenier formula} explain the relationship between $W_2$ distance and $\dot{H}^{-1}$ norm. For measure $\mu$ on $M$, let $\mathrm{d}\mu$ denote an infinitesimally small perturbation on $\mu$. One has
\begin{equation}
    W_2(\mu,\mu+\mathrm{d}\mu) = \|\mathrm{d}\mu\|_{\dot{H}^{-1}(\mu)} + o(\mathrm{d}\mu)
\end{equation}
i.e., for two measures $\mu,\nu$ on $M$,
\begin{equation}
\mathrm{W}_2(\mu, \nu)=\inf \left\{\int_0^1\|\mathrm{~d} \mu\|_{\dot{\mathrm{H}}^{-1}(\mu_t)} \mid \mu_0=\mu, \mu_1=\nu\right\} .
\end{equation}

There exists an equivalence relation between $W^2$ distance and $\dot{H}^{-1}$ norm. According to \cite{loeper2006uniqueness} we have the following result.  
\begin{lemma}\label{H1-W2}
For any positive measure $\mu$ and $\nu$ on $M$, denote $\lambda$ standard measure provided by the volume form (in the case of $M=\mathbb{R}^n$, $\lambda$ is the Lebesgue measure). If $\mu,\nu\leq \beta\lambda$ for some $\beta\geq 0$, then
\[
    \|\mu-\nu\|_{\dot{\mathrm{H}}^{-1}} \leq \sqrt{\beta} \mathrm{W}_2(\mu, \nu)
\]
\end{lemma}

\subsection{FMRC error}\label{subsec:fmrc-error}
Returning to the FMRC method, we first prove the following results.

\begin{lemma}\label{drift_W2}
If the true and approximate velocity fields \( v_{0}^{\mathrm{RC},*} \) and \( v_{0,\theta}^{\mathrm{RC}} \) in FMRC satisfy
\[
\int_{0}^{1}\mathbb{E}\left[\left\Vert v_{0}^{\mathrm{RC},*}(s,X_{t+\tau}^{s},X_{t})-v_{0,\theta}^{\mathrm{RC}}(s,X_{t+\tau}^{s},X_{t})\right\Vert ^{2}\right]\mathrm{d}s\leq\varepsilon^{2},
\]
and \( v_{0,\theta}^{\mathrm{RC}}(s, y^s, x) \) is \( L_s \)-Lipschitz in \((y^s, x)\), then the \( W_2 \) distance between joint distributions of \((X_t, X_{t+\tau})\) and \((X_t, \hat{X}_{t+\tau})\) is bounded by \( \varepsilon e^{\int_0^1 L_s \mathrm{d} s} \).
\end{lemma}

\begin{proof}
Define the augmented velocity fields:
\[
\bar{v}_{0}^{\mathrm{RC},*}(s, y^s, x) = \left(\begin{array}{c}
0\\
v_{0}^{\mathrm{RC},*}(s, y^s, x)
\end{array}\right),
\quad
\bar{v}_{0,\theta}^{\mathrm{RC}}(s, y^s, x) = \left(\begin{array}{c}
0\\
v_{0,\theta}^{\mathrm{RC}}(s, y^s, x)
\end{array}\right).
\]
It is straightforward to verify that the flow models defined by \(\bar{v}_{0}^{\mathrm{RC},*}\) and \(\bar{v}_{0,\theta}^{\mathrm{RC}}\) can transform the distribution of $(X_t, X_{t+\tau}^\prime)$ into distributions of $(X_t, X_{t+\tau})$ and $(X_t,\hat X_{t+\tau})$. Then, according to \cite{benton2023error}, we can prove the conclusion.
\end{proof}

\begin{lemma}\label{drift1-W2}
If
\[
\int_{0}^{1}\mathbb{E}\left[\left\Vert v_{1}^{\mathrm{RC},*}(s,X_{t}^{s},X_{t+\tau})-v_{1,\theta}^{\mathrm{RC}}(s,X_{t}^{s},X_{t+\tau})\right\Vert ^{2}\right]\mathrm{d}s\leq\varepsilon^{2},
\]
and $v_{1,\theta}^{\mathrm{RC}}(s,x^{s},y)$ is $L_{s}$-Lipschitz
in $(x^{s},y)$, then the $W_{2}$ distance between $(X_{t},X_{t+\tau})$
and $(\hat{X}_{t},X_{t+\tau})$ is also bounded by $\varepsilon e^{\int_{0}^{1}L_{s}\mathrm{d}s}$. 
\end{lemma}
\begin{proof}
The proof is similar to that of Lemma \ref{drift_W2} and is therefore omitted.
\end{proof}

We then provide our main result on errors of transfer operators.

\begin{theorem}\label{W2-mainthm}
Suppose the velocity fields satisfy conditions in \cref{drift_W2} and \cref{drift1-W2}, with velocity fields error $\varepsilon,\  \Tilde{\varepsilon}$ and Lipschitz constants $L_s,\ \Tilde{L}_s$. Suppose the distributions of $(X_t,X_{t+\tau})$ 
and its approximation $(X_t,\hat{X}_{t+\tau})$, as well as $(X_t,X_{t+\tau})$ and $(\hat{X}_t,X_{t+\tau})$, both satisfy conditions in \cref{H1-W2} with constants $\beta, \Tilde{\beta}$. We denote by $H^1_\rho$ the standard Sobolev space with norm
\[
\|f(x)\|_{H^1_\rho} = \left(\int \left(\|f(x)\|^2+\|\nabla f(x)\|^2\right)\rho(x)\mathrm{d}x\right)^{1/2}
\]
Define the weak form of operator error by inner product
\begin{equation}
    \begin{gathered}
        \|\mathcal{K}_\tau -\hat{\mathcal{K}}_L\|_H := \sup_{\|g(x)\|_{H^1_{\rho_0}}\leq 1, \|h(y)\|_{H^1_{\rho_1}\leq 1}}  \left<g(x), \left(\mathcal{K}_\tau h-\hat{\mathcal{K}}_L h \right)(x)\right>_{\rho_0}\\
        \|\mathcal{T}_\tau -\hat{\mathcal{T}}_D\|_H := \sup_{\|\Tilde{g}(x)\|_{H^1_{\rho_0}\leq 1}, \|\Tilde{h}(y)\|_{H^1_{\rho_1}}\leq 1}  \left<\Tilde{h}(y), \left(\mathcal{T}_\tau \Tilde{g}-\hat{\mathcal{T}}_L \Tilde{g}\right)(y)\right>_{\rho_1}
    \end{gathered}
\end{equation}
where $h,\Tilde{h}\in H_{\rho_1}^1,\ \Tilde{g},g\in H_{\rho_0}^1$. The operator errors are
\[
\begin{aligned}
    \|\mathcal{K}_\tau -\hat{\mathcal{K}}_L\|_H \leq \sqrt{\beta}\varepsilon e^{\int_0^1 L_s \mathrm{~d} s}\\
    \|\mathcal{T}_\tau -\hat{\mathcal{T}}_D\|_H \leq \sqrt{\Tilde{\beta}}\Tilde{\varepsilon} e^{\int_0^1 \Tilde{L}_s \mathrm{~d} s}
\end{aligned}
\]
\end{theorem}
\begin{proof}
    See \cref{W2proof}.
\end{proof} 

\section{Numerical examples}
As a verification of the effectiveness of our method, we consider a drift-diffusion process governed by the following SDE:
\begin{equation}
\mathrm{d} X_t=-\nabla V\left(X_t\right) \mathrm{d} t+\sqrt{2 \beta^{-1}} \mathrm{~d} W_t
\end{equation}
Here $V$ is called the potential, $\beta$ is the inverse temperature, and $W_t$ denotes a standard Wiener process. We define $V(x)$ here to be a three-dimensional potential
\begin{equation}
    \begin{aligned}
    &V(x_1,x_2,x_3) = V'(x_1,x_2) + 10x_3^2\\
    &V'(x_1,x_2) = \cos(7\arctan(x_2,x_1)) + 10(\sqrt{x_1^2+x_2^2}-1)^2
    \end{aligned}
\end{equation}
$V'(x_1,x_2)$ represents a 2D circular potential with seven wells, while $x_3$ evolves as an independent Ornstein-Uhlenbeck process with equilibrium distribution $\mathcal{N}(x_3\mid0,0.05)$ and a mixing time much shorter than that of $(x_1,x_2)$. The data are simulated by the Euler-Maruyama method with a time step $\Delta t=0.001$ then transformed by a nonlinear "Swiss roll map". \cref{pic:7well_energy} shows the 2D circular potential. \cref{pic:cluster}, \cref{pic:traj_3d} and \cref{pic:traj_swiss} illustrate the process of how we ultimately obtain the data in the "Swiss roll" configuration. 
\begin{figure}[!t]
    \centering
    \begin{subfigure}{0.49\textwidth}
        \centering
        \includegraphics[width=0.9\textwidth]{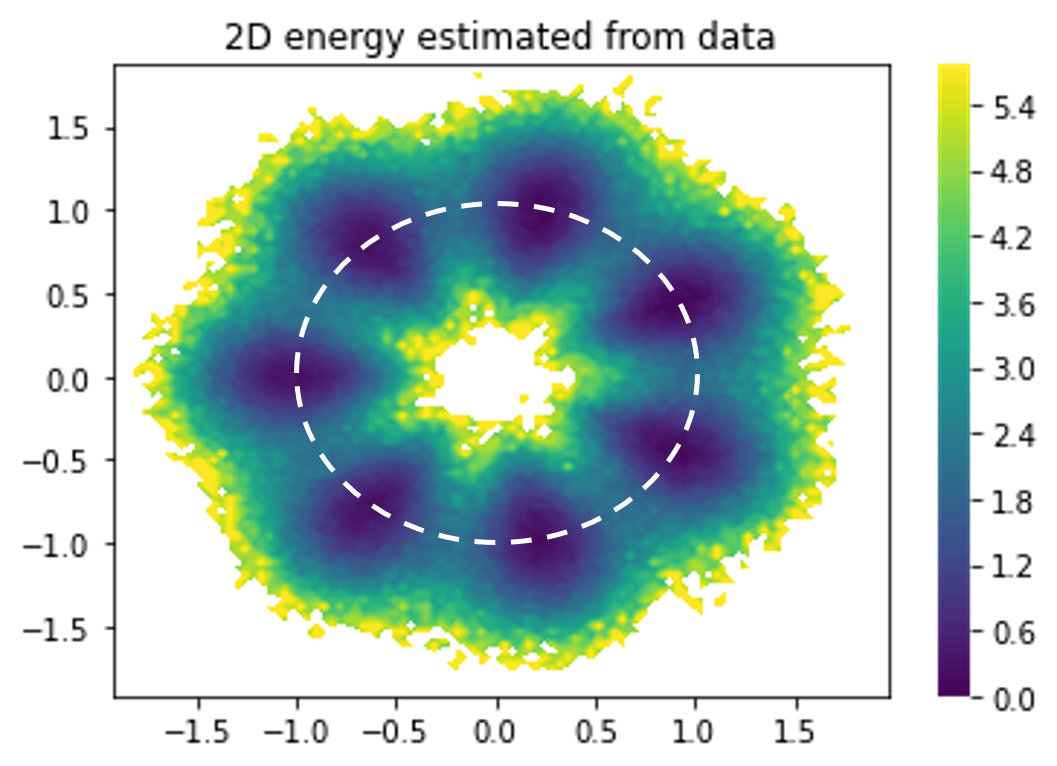}
        \caption{}
        \label{pic:7well_energy}
    \end{subfigure}
    \begin{subfigure}{0.49\textwidth}
        \centering
        \includegraphics[width=0.9\textwidth]{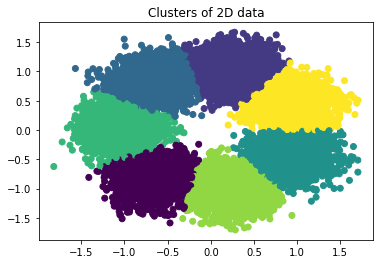}
        \caption{}
        \label{pic:cluster}
    \end{subfigure}
    
    \begin{subfigure}{0.49\textwidth}
        \centering
        \includegraphics[width=0.9\textwidth]{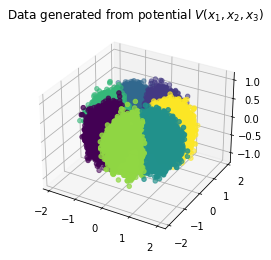}
        \caption{}
        \label{pic:traj_3d}
    \end{subfigure}
    \begin{subfigure}{0.49\textwidth}
        \centering
        \includegraphics[width=0.9\textwidth]{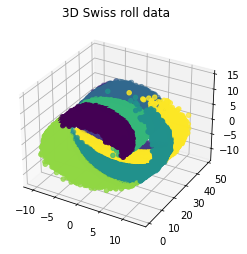}
        \caption{}
        \label{pic:traj_swiss}
    \end{subfigure}
    \caption{(a) Energy landscape: circular potential with seven wells. (b) Clustering of 2D data, generated by potential $V'(x_1,x_2)$. (c) Corresponding 3D data generated by full potential $V(x_1,x_2,x_3)$. (d) Swiss roll data, by applying a nonlinear transformation to data in (c).}
\end{figure}

We use Multi-Layer Perceptrons (MLPs) to model both the reaction coordinate and the velocity fields. We set the dimension of the reduced process to one and applied the FMRC method to obtain an optimal reaction coordinate $r=r(\mathbf{x})$. Our results demonstrate that the corresponding RC effectively distinguishes metastable states while preserving the transition dynamics.
\begin{figure}[!t]
    \centering
    \begin{subfigure}{0.49\textwidth}
        \centering
        \includegraphics[width=0.9\textwidth]{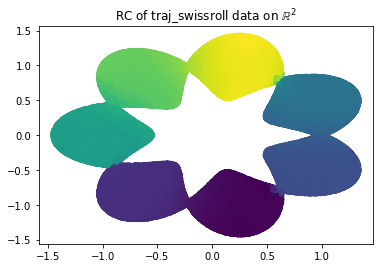}
        \caption{}
        \label{pic:RC_2d}
    \end{subfigure}
    \hfill
    \begin{subfigure}{0.49\textwidth}
        \centering
        \includegraphics[width=0.9\textwidth]{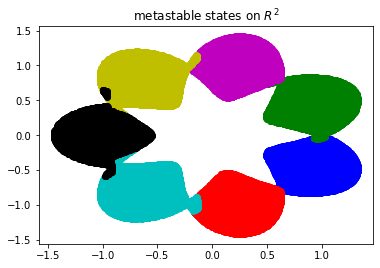}
        \caption{}
        \label{pic:meta_3d}
    \end{subfigure}
    \hfill
    \begin{subfigure}{0.49\textwidth}
        \centering
        \includegraphics[width=\textwidth]{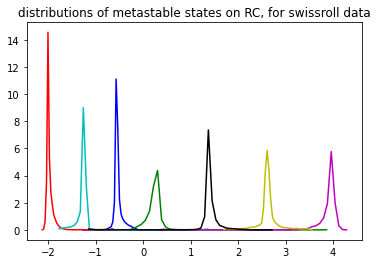}
        \caption{}
        \label{pic:RC_distribution}
    \end{subfigure}
    \caption{(a) Plot RC $r(x)$ onto 2D projected data. The values of the reaction coordinate distinguish the different wells. (b) Apply PCCA+ to the projected data to obtain clustering results and analyze the values of the reaction coordinate within each cluster. (c) The histogram of the RC values, categorized by the PCCA+ results, shows clear gaps between different clusters.}
\end{figure}

It is known that the seven local minima of $V(x)$ induce metastability in the system. This
implies that a typical trajectory will oscillate around a minimum for an extended period before suddenly, due to sufficiently strong stochastic excitation, undergoing a rapid transition to another local minimum. Both the additional dimension from the Ornstein-Uhlenbeck (OU) process and Swiss-roll transformation do not alter these dynamics. Therefore, our reaction coordinate should distinguish data originating from different wells.

\cref{pic:RC_2d} displays the values of $r(x)$ plotted on $\mathbb{R}^2$. The values of the reaction coordinate clearly distinguish the wells. To further illustrate this, we cluster the 2D data from the circular potential and examine the RC values within each cluster. \cref{pic:meta_3d} shows the clustering results by using Markov state model and PCCA+ \cite{roblitz2013fuzzy}. \cref{pic:RC_distribution} shows that the values of $r(x)$ are consistent within each cluster, yet differ between clusters. This indicates that the reaction coordinate learned from the FMRC method is effective to some extent.

\section{Conclusions}
\label{sec:conclusions}
In this paper, we introduce a novel model called FMRC for identifying optimal reaction coordinates We also highlight that our method implicitly defines reduced-order transfer operators. We present the specific form of these low-dimensional operators and prove that the corresponding operator error is small in terms of an inner-product-based operator norm.

\section*{Acknowledgments}
The authors would like to express their gratitude to Xin Zhang, Mingyuan Zhang and Yajun Wang for their insightful discussions.

\FloatBarrier

\appendix
\section{Proof of \cref{W2-mainthm}}\label{W2proof}
\begin{proof}
Here, we present the proof of the error derived from the reduced-order Koopman operator. The proof for the PF operator follows the same approach. We denote $\hat{p}_\tau^L(x,y)$ be transition densities approximated by using $v_{0,\theta}^{\mathrm{RC}}$ in FMRC. 

Utilizing the equivalence between the $W_2$ distance and the $\dot{H}^{-1}$ norm, we can obtain
\begin{align*}
    W_2(&\rho_0(x)p_\tau(x,y),\rho_0(x)\hat{p}_\tau^L(x,y)) \geq \frac{1}{\sqrt{\beta}}\| \rho_0(x)p_\tau(x,y)-\rho_0(x)\hat{p}^L_\tau(x,y) \|_{\dot{H}^{-1}(\mathbb{R}^{2D})}\\ 
    &\geq \frac{1}{\sqrt{\beta}} \sup_{\|f\|_{\dot{H}^1(\rho(x,y))}\leq 1} \left< f(x,y),\rho_0(x)p_\tau(x,y)-\rho_0(x)\hat{p}^L_\tau(x,y) \right>\\
    &\geq \frac{1}{\sqrt{\beta}} \sup_{\|g(x)h(y))\|_{\dot{H}^1(\rho(x,y))}\leq 1} \left< g(x)h(y),\rho_0(x)p_\tau(x,y)-\rho_0(x)\hat{p}^L_\tau(x,y) \right>\\
    &\geq \frac{1}{\sqrt{\beta}} \sup_{\|g(x)\|_{H^1(\rho_0(x))}\leq 1, \|h(y)\|_{H^1(\rho_1(y))}\leq 1} \left< g(x)h(y),\rho_0(x)p_\tau(x,y)-\rho_0(x)\hat{p}^L_\tau(x,y) \right>\\
    &= \frac{1}{\sqrt{\beta}} \sup_{\|g(x)\|_{H^1(\rho_0(x))}\leq 1, \|h(y)\|_{H^1(\rho_1(y))}\leq 1}  \int_X \int_Y g(x)h(y)\left(\rho_0(x)p_\tau(x,y)-\rho_0(x)\hat{p}^L_\tau(x,y)\right)\mathrm{d}x\mathrm{d}y \\
    &= \frac{1}{\sqrt{\beta}} \sup_{\|g(x)\|_{H^1(\rho_0(x))}\leq 1, \|h(y)\|_{H^1(\rho_1(y))}\leq 1}  \int_X g(x)\rho_0(x) \left(\int_Y \left(h(y)p_\tau(x,y)-h(y)\hat{p}^L_\tau(x,y)\right)\mathrm{d}y\right) \mathrm{d}x\\
    &= \frac{1}{\sqrt{\beta}} \sup_{\|g(x)\|_{H^1(\rho_0(x))}\leq 1, \|h(y)\|_{H^1(\rho_1(y))}\leq 1} \left<g(x), \mathcal{K}_\tau h-\hat{\mathcal{K}}_L h\right>_{\rho_0}\\
    &= \frac{1}{\sqrt{\beta}}\|\mathcal{K}_\tau -\hat{\mathcal{K}}_L\|_H
\end{align*}

For the reduced-order Perron-Frobenius (PF) operator, the proof follows the same approach
\begin{align*}
    W_2(&\rho_0(x)p_\tau(x,y),\rho_0(x)\hat{p}_\tau^L(x,y)) \geq \frac{1}{\sqrt{\beta}}\| \rho_0(x)p_\tau(x,y)-\rho_0(x)\hat{p}^L_\tau(x,y) \|_{\dot{H}^{-1}(\mathbb{R}^{2D})}\\ 
    &\geq \frac{1}{\sqrt{\beta}} \sup_{\|f\|_{\dot{H}^1(\rho(x,y))}\leq 1} \left< f(x,y),\rho_0(x)p_\tau(x,y)-\rho_0(x)\hat{p}^L_\tau(x,y) \right>\\
    &\geq \frac{1}{\sqrt{\beta}} \sup_{\|g(x)h(y))\|_{\dot{H}^1(\rho(x,y))}\leq 1} \left< g(x)h(y),\rho_0(x)p_\tau(x,y)-\rho_0(x)\hat{p}^L_\tau(x,y) \right>\\
    &\geq \frac{1}{\sqrt{\beta}} \sup_{\|g(x)\|_{H^1(\rho_0(x))}\leq 1, \|h(y)\|_{H^1(\rho_1(y))}\leq 1} \left< g(x)h(y),\rho_0(x)p_\tau(x,y)-\rho_0(x)\hat{p}^L_\tau(x,y) \right>\\
    &= \frac{1}{\sqrt{\beta}} \sup_{\|g(x)\|_{H^1(\rho_0(x))}\leq 1, \|h(y)\|_{H^1(\rho_1(y))}\leq 1}  \int_X \int_Y g(x)h(y)\left(\rho_0(x)p_\tau(x,y)-\rho_0(x)\hat{p}^L_\tau(x,y)\right)\mathrm{d}x\mathrm{d}y \\
    &= \frac{1}{\sqrt{\beta}} \sup_{\|g(x)\|_{H^1(\rho_0(x))}\leq 1, \|h(y)\|_{H^1(\rho_1(y))}\leq 1}  \int_Y h(y)\rho_1(y) \int_X \frac{\rho_0(x)}{\rho_1(y)}\left( p_\tau(x,y)-\hat{p}^L_\tau(x,y) \right)g(x)\mathrm{d}x \mathrm{d}y\\
    &= \frac{1}{\sqrt{\beta}} \sup_{\|g(x)\|_{H^1(\rho_0(x))}\leq 1, \|h(y)\|_{H^1(\rho_1(y))}\leq 1} \left<h(y), \left(\mathcal{T}_\tau g-\hat{\mathcal{T}}_D g\right)(y)\right>_{\rho_1}\\
    &= \frac{1}{\sqrt{\beta}}\|\mathcal{T}_\tau -\hat{\mathcal{T}}_D\|_H
\end{align*}

The detailed proof that the condition $\|f\|_{\dot{H}^1(\rho(x,y))}\leq 1$ can be transformed into the conditions 
\[
\|g(x)\|_{H^1(\rho_0(x))}\leq 1,\quad \|h(y)\|_{H^1(\rho_1(y))}\leq 1
\] 
is as follows.
\begin{align*}
    &\|f\|_{\dot{H}^1(\rho(x,y))}\leq 1 \Longleftrightarrow \int_{X\times Y}\|\nabla f(x,y)\|^2_{\mathbb{R}^{2n}}\rho(x,y)\mathrm{d}x\mathrm{d}y\leq 1\\
    & \Longrightarrow \int_{X\times Y}\|\nabla\cdot \left(g(x)h(y)\right)\|^2_{\mathbb{R}^{2n}}\rho(x,y)\mathrm{d}x\mathrm{d}y\leq 1\\
    & \Longrightarrow \int_{X\times Y}\left(\|h(y)\nabla_x g(x)\|^2_{\mathbb{R}^n}+\|g(x)\nabla_y h(y)\|^2_{\mathbb{R}^n}\right)\rho(x,y)\mathrm{d}x\mathrm{d}y\leq 1\\
    & \Longrightarrow \int_{X\times Y}h^2(y)\|\nabla_x g(x)\|^2_{\mathbb{R}^n}\rho(x,y)\mathrm{d}x\mathrm{d}y + \int_{X\times Y}g^2(x)\|\nabla_y h(y)\|^2_{\mathbb{R}^n}\rho(x,y)\mathrm{d}x\mathrm{d}y \leq 1\\
    & \Longrightarrow \|g(x)\|_{H^1(\rho_0(x))}\leq 1,\quad \|h(y)\|_{H^1(\rho_1(y))}\leq 1
\end{align*}

\end{proof}

\bibliographystyle{siamplain}
\bibliography{references}
\end{document}

%% file: ex_shared.tex
\usepackage{lipsum}
\usepackage{amsfonts}
\usepackage{graphicx}
\usepackage{epstopdf}
\usepackage{algorithmic}
\ifpdf
  \DeclareGraphicsExtensions{.eps,.pdf,.png,.jpg}
\else
  \DeclareGraphicsExtensions{.eps}
\fi

\newsiamremark{remark}{Remark}
\newsiamremark{hypothesis}{Hypothesis}
\crefname{hypothesis}{Hypothesis}{Hypotheses}
\newsiamthm{claim}{Claim}

\headers{Flow Matching for Transfer Operator Approximation}{Zhicheng Zhang, Ling Guo, and Hao Wu}

\title{Optimal Low-dimensional Approximation of Transfer Operators via Flow Matching: Computation and Error Analysis}

\author{Zhicheng Zhang\thanks{School of Mathematical Sciences, Tongji University, Shanghai, P. R. China.}
\and Ling Guo\thanks{Department of Mathematics, Shanghai Normal University, Shanghai, P. R. China.}
\and Hao Wu\thanks{School of Mathematical Sciences, Institute of Natural Sciences and MOE-LSC, Shanghai Jiao Tong University, Shanghai, P. R. China 
  (\email{hwu81@sjtu.edu.cn}).}
}

\usepackage{amsopn}


%% file: main_part_FMRC.bbl
\begin{thebibliography}{10}

\bibitem{albergo2023stochastic}
{\sc M.~S. Albergo, N.~M. Boffi, and E.~Vanden-Eijnden}, {\em Stochastic interpolants: A unifying framework for flows and diffusions}, arXiv preprint arXiv:2303.08797,  (2023).

\bibitem{benton2023error}
{\sc J.~Benton, G.~Deligiannidis, and A.~Doucet}, {\em Error bounds for flow matching methods}, arXiv preprint arXiv:2305.16860,  (2023).

\bibitem{bittracher2018transition}
{\sc A.~Bittracher, P.~Koltai, S.~Klus, R.~Banisch, M.~Dellnitz, and C.~Sch{\"u}tte}, {\em Transition manifolds of complex metastable systems: Theory and data-driven computation of effective dynamics}, Journal of nonlinear science, 28 (2018), pp.~471--512.

\bibitem{bittracher2023optimal}
{\sc A.~Bittracher, M.~Mollenhauer, P.~Koltai, and C.~Sch{\"u}tte}, {\em Optimal reaction coordinates: Variational characterization and sparse computation}, Multiscale Modeling \& Simulation, 21 (2023), pp.~449--488.

\bibitem{chen2019nonlinear}
{\sc W.~Chen, H.~Sidky, and A.~L. Ferguson}, {\em Nonlinear discovery of slow molecular modes using state-free reversible vampnets}, The Journal of chemical physics, 150 (2019).

\bibitem{federici2023latent}
{\sc M.~Federici, P.~Forr{\'e}, R.~Tomioka, and B.~S. Veeling}, {\em Latent representation and simulation of markov processes via time-lagged information bottleneck}, arXiv preprint arXiv:2309.07200,  (2023).

\bibitem{lipman2022flow}
{\sc Y.~Lipman, R.~T. Chen, H.~Ben-Hamu, M.~Nickel, and M.~Le}, {\em Flow matching for generative modeling}, arXiv preprint arXiv:2210.02747,  (2022).

\bibitem{liu2022rectified}
{\sc Q.~Liu}, {\em Rectified flow: A marginal preserving approach to optimal transport}, arXiv preprint arXiv:2209.14577,  (2022).

\bibitem{loeper2006uniqueness}
{\sc G.~Loeper}, {\em Uniqueness of the solution to the vlasov--poisson system with bounded density}, Journal de math{\'e}matiques pures et appliqu{\'e}es, 86 (2006), pp.~68--79.

\bibitem{mardt2018vampnets}
{\sc A.~Mardt, L.~Pasquali, H.~Wu, and F.~No{\'e}}, {\em Vampnets for deep learning of molecular kinetics}, Nature communications, 9 (2018), p.~5.

\bibitem{noe2013variational}
{\sc F.~No{\'e} and F.~Nuske}, {\em A variational approach to modeling slow processes in stochastic dynamical systems}, Multiscale Modeling \& Simulation, 11 (2013), pp.~635--655.

\bibitem{nuske2016variational}
{\sc F.~N{\"u}ske, R.~Schneider, F.~Vitalini, and F.~No{\'e}}, {\em Variational tensor approach for approximating the rare-event kinetics of macromolecular systems}, The Journal of chemical physics, 144 (2016).

\bibitem{peyre2018comparison}
{\sc R.~Peyre}, {\em Comparison between $w_2$ distance and $\dot{H}^{- 1}$ norm, and localization of wasserstein distance}, ESAIM: Control, Optimisation and Calculus of Variations, 24 (2018), pp.~1489--1501.

\bibitem{roblitz2013fuzzy}
{\sc S.~R{\"o}blitz and M.~Weber}, {\em Fuzzy spectral clustering by pcca+: application to markov state models and data classification}, Advances in Data Analysis and Classification, 7 (2013), pp.~147--179.

\bibitem{tong2023improving}
{\sc A.~Tong, N.~Malkin, G.~Huguet, Y.~Zhang, J.~Rector-Brooks, K.~Fatras, G.~Wolf, and Y.~Bengio}, {\em Improving and generalizing flow-based generative models with minibatch optimal transport}, arXiv preprint arXiv:2302.00482,  (2023).

\bibitem{villani2009optimal}
{\sc C.~Villani et~al.}, {\em Optimal transport: old and new}, vol.~338, Springer, 2009.

\bibitem{wang2021state}
{\sc D.~Wang and P.~Tiwary}, {\em State predictive information bottleneck}, The Journal of Chemical Physics, 154 (2021).

\bibitem{wang2019past}
{\sc Y.~Wang, J.~M.~L. Ribeiro, and P.~Tiwary}, {\em Past--future information bottleneck for sampling molecular reaction coordinate simultaneously with thermodynamics and kinetics}, Nature communications, 10 (2019), p.~3573.

\bibitem{wehmeyer2018time}
{\sc C.~Wehmeyer and F.~No{\'e}}, {\em Time-lagged autoencoders: Deep learning of slow collective variables for molecular kinetics}, The Journal of chemical physics, 148 (2018).

\bibitem{williams2015data}
{\sc M.~O. Williams, I.~G. Kevrekidis, and C.~W. Rowley}, {\em A data--driven approximation of the koopman operator: Extending dynamic mode decomposition}, Journal of Nonlinear Science, 25 (2015), pp.~1307--1346.

\bibitem{wu2020variational}
{\sc H.~Wu and F.~No{\'e}}, {\em Variational approach for learning markov processes from time series data}, Journal of Nonlinear Science, 30 (2020), pp.~23--66.

\bibitem{wu2024reaction}
{\sc H.~Wu and F.~No{\'e}}, {\em Reaction coordinate flows for model reduction of molecular kinetics}, The Journal of Chemical Physics, 160 (2024).

\end{thebibliography}
